\documentclass[12pt]{amsart}
\usepackage{latexsym, amsthm, amscd, euscript, float, subfig}
\usepackage{color}
\usepackage{cite}
\usepackage{amssymb}
\usepackage{amsmath}
\usepackage{graphicx}
\usepackage[normalem]{ulem}
\usepackage{multirow}

\usepackage{array}

\newtheorem{theorem}{Theorem}[section]
\newtheorem{lemma}[theorem]{Lemma}
\newtheorem{corollary}[theorem]{Corollary}
\newtheorem{proposition}[theorem]{Proposition}
\newtheorem{ca}{Case}

\newtheorem{remark}[theorem]{Remark}

\newtheorem*{acknowledgement}{Acknowledgment}

\newtheorem{conjecture}[theorem]{Conjecture}

\numberwithin{equation}{section}

\newcounter{minutes}
\divide\time by 60
\newcounter{hours}
\multiply\time by 60
\addtocounter{minutes}{-\time}

\newcommand{\IR}{\mathbb{R}}
\newcommand{\IB}{\mathbb{B}}
\newcommand{\ds}{\displaystyle}

\newcommand{\Bn}{ {\mathbb{B}^n} }
\newcommand{\Rn}{ {\mathbb{R}^n} }

\renewcommand{\sinh}{\,\textnormal{sinh}}

\def\be{\begin{equation}}
\def\ee{\end{equation}}

\newcommand{\bee}{\begin{enumerate}}
\newcommand{\eee}{\end{enumerate}}

\newcommand{\blem}{\begin{lem}}
\newcommand{\elem}{\end{lem}}
\newcommand{\bthm}{\begin{thm}}
\newcommand{\ethm}{\end{thm}}
\newcommand{\bcor}{\begin{cor}}
\newcommand{\ecor}{\end{cor}}
\newcommand{\beg}{\begin{exam}}
\newcommand{\eeg}{\end{exam}}
\newcommand{\begs}{\begin{examples}}
\newcommand{\eegs}{\end{examples}}
\newcommand{\bdefe}{\begin{defn}}
\newcommand{\edefe}{\end{defn}}
\newcommand{\bprob}{\begin{prob}}
\newcommand{\eprob}{\end{prob}}
\newcommand{\bques}{\begin{ques}}
\newcommand{\eques}{\end{ques}}
\newcommand{\bei}{\begin{itemize}}
\newcommand{\eei}{\end{itemize}}
\newcommand{\bcon}{\begin{conj}}
\newcommand{\econ}{\end{conj}}
\newcommand{\bop}{\begin{op}}
\newcommand{\eop}{\end{op}}

\newcommand{\bca}{\begin{ca}}
\newcommand{\eca}{\end{ca}}
\newcommand{\bsca}{\begin{sca}}
\newcommand{\esca}{\end{sca}}

\newcommand{\bcl}{\begin{cl}}
\newcommand{\ecl}{\end{cl}}

\newcommand{\bscl}{\begin{scl}}
\newcommand{\escl}{\end{scl}}

\newcommand{\bcons}{\begin{conjs}}
\newcommand{\econs}{\end{conjs}}
\newcommand{\bprop}{\begin{propo}}
\newcommand{\eprop}{\end{propo}}
\newcommand{\br}{\begin{rem}}
\newcommand{\er}{\end{rem}}
\newcommand{\brs}{\begin{rems}}
\newcommand{\ers}{\end{rems}}
\newcommand{\bo}{\begin{obser}}
\newcommand{\eo}{\end{obser}}
\newcommand{\bos}{\begin{obsers}}
\newcommand{\eos}{\end{obsers}}
\newcommand{\bpf}{\begin{pf}}
\newcommand{\epf}{\end{pf}}
\newcommand{\ba}{\begin{array}}
\newcommand{\ea}{\end{array}}
\newcommand{\beq}{\begin{eqnarray}}
\newcommand{\beqq}{\begin{eqnarray*}}
\newcommand{\eeq}{\end{eqnarray}}
\newcommand{\eeqq}{\end{eqnarray*}}

\thanks{* The corresponding author}

\begin{document}
\vspace*{-2cm}
\title[A Gromov hyperbolic metric vs the hyperbolic and other related metrics]
{A Gromov hyperbolic metric vs the hyperbolic and other related metrics}

\def\thefootnote{}
\footnotetext{ \texttt{\tiny File:~\jobname .tex,
          printed: \number\day-\number\month-\number\year,
          \thehours.\ifnum\theminutes<10{0}\fi\theminutes}
} \makeatletter\def\thefootnote{\@arabic\c@footnote}\makeatother

\author[M. R. Mohapatra]{Manas Ranjan Mohapatra}
\address{Manas Ranjan Mohapatra, Discipline of Mathematics,
Indian Institute of Technology Indore,
Simrol, Khandwa Road, Indore 453 552, India
}
\email{mrm.iiti@gmail.com}

\author[S. K. Sahoo]{Swadesh Kumar Sahoo$^{*}$}
\address{Swadesh Kumar Sahoo, Discipline of Mathematics,
Indian Institute of Technology Indore,
Simrol, Khandwa Road, Indore 453 552, India}
\email{swadesh@iiti.ac.in}

\begin{abstract}
We mainly consider two metrics: a Gromov hyperbolic metric and a scale invariant Cassinian metric. We 
compare these two metrics and obtain their relationship
with certain well-known hyperbolic-type metrics, leading to several inclusion relations 
between the associated metric balls. 
\\


\end{abstract}

\maketitle
\thispagestyle{empty}
\section{Introduction}
Comparison of hyperbolic-type metrics, defined over proper subdomains of $\Rn$, is a significant part of geometric function theory as it reveals the geometric property of the domain. For example, uniform 
domains are defined by comparing the quasihyperbolic metric \cite{GP76} and the distance ratio metric \cite{GO79}. 
In recent years, many authors have contributed to
the study of hyperbolic-type metrics. Some of the familiar hyperbolic-type metrics are
the Apollonian metric \cite{Bea98,BI08,Has03,Ibr03}, the half-Apollonian metric \cite{HL04}, the Seittenranta metric \cite{Sei99}, the Cassinian metric \cite{Ibr09,IMSZ,KMS,HKVZ},
the triangular ratio metric \cite{CHKV15}, etc. These metrics are also referred as the {\em relative metrics} since they are defined in proper subdomains of the Euclidean space $\Rn$, $n\ge 2$, relative to domain boundaries.
A general form of some of these relative metrics has been considered by P. H\"ast\"o in \cite[Lemma~6.1]{Has02}, in a different context.

Very recently, a scale invariant version of the Cassinian metric has been studied
by Ibragimov in \cite{Ibr16} which is defined by
$$\tilde{\tau}_D(x,y)=\log\left(1+\sup_{p\in \partial D}\frac{|x-y|}{\sqrt{|x-p||p-y|}}\right), \quad x,y\in D\subsetneq \Rn.
$$
The interesting part of this metric is that many
properties in arbitrary domains are revealed in the setting of once-punctured spaces.
For example, $\tilde{\tau}_D$ is a metric in an arbitrary domain $D\subsetneq \Rn$ if it is a metric
on once-punctured spaces. The $\tilde{\tau}_D$-metric is comparable with the Vuorinen's distance ratio metric \cite{Vuo85} in
arbitrary domains $D\subsetneq \Rn$ if they are comparable in the punctured spaces (see \cite{Ibr16}). 
It is appropriate here to recall that the $\tilde{\tau}_D$-metric satisfies the domain monotonicity property, i.e. if $D_1,D_2\subsetneq \Rn$ with $D_1\subset D_2$, then 
$\tilde{\tau}_{D_2}(x,y)\le \tilde{\tau}_{D_1}(x,y)$ for all $x,y\in D_1$ \cite[p.~2]{Ibr16}.

Gromov in 1987 introduced the notion of an abstract hyperbolic space \cite{Gro87}. 
Let $(D,d)$ be
a metric space and $x,y,z\in D$. The {\em Gromov product}\index{Gromov product} of $x$ and $y$ with respect to $z$
is defined by the formula
$$(x|y)_z=\frac{1}{2} \left[d(x,z)+d(y,z)-d(x,y)\right].
$$
The metric space $(D,d)$ is said to be {\em Gromov hyperbolic}\index{Gromov hyperbolic} if there exists $\beta \ge 0$ such that
$$(x|y)_w\ge (x|z)_w \wedge (z|y)_w -\beta
$$
for all $x,y,z,w\in D$. We also say that $D$ is $\beta$-hyperbolic. Equivalently, 
the metric space $(D,d)$ is called Gromov hyperbolic if and only if there exist a constant $\beta>0$ such that
$$d(x,z)+d(y,w)\le (d(x,w)+d(y,z))\vee (d(x,y)+d(z,w))+2\beta
$$
for all points $x,y,z,w \in D$. Note that Gromov hyperbolicity is preserved under
quasi-isometries\index{quasi-isometry}. That means it is preserved under the mappings $f:(D,d_1)\to (f(D),d_2)$ satisfying
$$\frac{1}{\lambda} d_1(x,y)-k\le d_2(f(x),f(y))\le \lambda d_1(x,y)+k, \quad x,y\in D,
$$ 
where $\lambda \ge 1$ and $k\ge 0$. Literature on Gromov hyperbolicity are available in
\cite{BS00,Gro87,Has05,Ibr11,Ibr11-I,Vai05}.
One natural question was to investigate {\em whether a metric space is hyperbolic in the
sense of Gromov or not}? 
Ibragimov in \cite{Ibr11} introduced a metric, $u_Z$, which hyperbolizes 
(in the sense of Gromov) the locally compact non-complete metric space $(Z,d)$ without changing its quasiconformal geometry, by 
$$u_Z(x,y)=2\log \frac{d(x,y)+\max\{{\rm dist(x, \partial Z)},{\rm dist(y, \partial Z)}\}}{\sqrt{{\rm dist(x, \partial Z)}\,{\rm dist(y, \partial Z)}}}, \quad x,y\in Z.
$$ 
For a domain $D\subsetneq \Rn$ equipped with the Euclidean metric, the {\em $u_D$-metric} is defined by
$$u_D(x,y)=2\log \frac{|x-y|+\max\{{\rm dist}(x,\partial D),{\rm dist}(y,\partial D)\}}
{\sqrt{{\rm dist}(x,\partial D)\,{\rm dist}(y,\partial D)}}, \quad x,y\in D.
$$
Note that the $u_D$-metric does not satisfy the domain monotonicity property and
it coincides with the Vuorinen's distance ratio metric in punctured spaces $\Rn\setminus\{p\}$, for $p\in\Rn$. 
Though comparisons of the $u_D$-metric with some hyperbolic-type metrics are studied in \cite{Ibr11}, 
in this paper, we further compare the $u_D$-metric with the hyperbolic metric and other related metrics which were not
considered in \cite{Ibr11}.

It is well known that the geometric structure of a metric space can be viewed from the geometric structure of the fundamental element, namely, the metric balls. Hence it is reasonable to study metric balls and their inclusion relations 
with other metric balls by fixing the centre common to each pair of metric balls. In this
regard we study inclusion relations associated with the metric balls.

The paper is organized as follows:
Section~\ref{prel} is devoted to the preliminaries for the upcoming sections. 
In Section~\ref{u-rho-bn}, we compare the $u_D$-metric with the hyperbolic metric of the 
unit ball and upper half space. In Section~\ref{u-tildetau}, we compare the $u_D$-metric
with the $\tilde{\tau}_D$-metric. Comparisons of the $u_D$-metric and the $\tilde{\tau}_D$-metric with other metrics are discussed in Section~\ref{u-tildetau-other}.

\section{Preliminaries}\label{prel}

Throughout the paper, $D$ denotes an arbitrary, proper subdomain of the Euclidean space $\IR^n$.
Symbolically, we write $D\subsetneq \IR^n$. 
Given $x\in\IR^n$ and $r>0$, the open ball centered at $x$ and of radius $r$ is denoted by $B(x,r)=\{y\in\IR^n\colon\, |x-y|<r\}$. 
Set $\IB^n=B(0,1)$. We denote by $D_p$ and $D_{p,q}$, the punctured spaces $\Rn\setminus\{p\}$ and $\Rn\setminus\{p,q\}$ respectively.
For a given $x\in D$, we set $d(x):={\rm dist}(x,\partial D)$.
For real numbers $r$ and $s$, we set $r\vee s=\max\{r,\, s\}$ and $r\wedge s=\min\{r,\, s\}$.

The {\it distance ratio metric}, $\tilde{j}_D$, is defined by
$$
\tilde{j}_D(x,y)=\log\left(1+\frac{|x-y|}{d(x)\wedge d(y)}\right), \quad x,y\in D.
$$
The above form of the metric $\tilde{j}_D$, which was first considered in \cite{Vuo85},
is a slight modification of the original distance ratio metric, $j_D$, of Gehring and Osgood \cite{GO79}, defined by 
$$j_D(x,y)=\frac{1}{2}\log \left(1+\frac{|x-y|}{d(x)}\right) \left(1+\frac{|x-y|}{d(y)}\right),\quad x,y\in D.
$$
The $\tilde{j}_D$-metric has been widely studied in the literature; see, for instance, \cite{Vuo88}. 
These two distance ratio metrics are related by: 
$\tilde{j}_D(x,y)/2 \le j_D(x,y)\le \tilde{j}_D(x,y)$; see \cite{Sei99,Has05}.

The {\it{hyperbolic metric}}, $\rho_{\IB^n}$, of the unit ball $\IB^n$ is given by
$$
\rho_{\Bn}(x,y)=\inf_\gamma\int_\gamma \frac{2|{\rm d}z|}{1-|z|^2},
$$
where the infimum is taken over all rectifiable curves $\gamma\subset \IB^n$
joining $x$ and $y$.

 Now,
we define the $d$-metric ball (metric ball with respect to the metric $d$) as follows: let $(D,d)$
be a metric space. Then the set 
$$B_d(x,R)=\{z\in D:\,d(x,z)<R\}
$$
is called the {\em $d$-metric ball} of the domain $D$. 

\section{Comparison of the $u_D$-metric with the hyperbolic metric}\label{u-rho-bn}
This section is devoted to the comparison of the $u_D$-metric and the $\rho_{D}$-metric 
when $D=\Bn$ and $D=\mathbb{H}^n:=\{(x_1,x_2,\ldots,x_{n}):x_n>0\}$.
The hyperbolic metric of the unit ball $\Bn$, $\rho_{\Bn}$, can be computed by the following formula 
(see \cite[p.~40]{Bea95}).
\begin{equation}\label{def-rho}
\sinh\left(\frac{\rho_{\Bn}(x,y)}{2}\right)=\frac{|x-y|}{\sqrt{(1-|x|^2)(1-|y|^2)}}.
\end{equation}

We first establish the relationship between the $u_\Bn$-metric and the $\rho_\Bn$-metric. 
In this setting, the following lemma is useful which yields a relationship between 
the $u_D$-metric and the $j_D$-metric in arbitrary subdomains of $\Rn$.

\begin{lemma}\label{lem-u-j}
Let $D\subsetneq \Rn$ be arbitrary. Then for $x,y\in D$ we have
$$2j_D(x,y)\le u_D(x,y)\le 4 j_D(x,y).
$$
The first inequality becomes equality when $d(x)=d(y)$. 
\end{lemma}
\begin{proof}
The first inequality is proved in \cite[Theorem~3.1]{Ibr11}. From the definitions
of the $j_D$-metric and the $u_D$-metric, it follows that $2j_D(x,y)=u_D(x,y)$ whenever $d(x)=d(y)$.
Now, we shall prove the second inequality. Without loss of generality we assume that $d(x)\ge d(y)$ for $x,y\in D\subsetneq \Rn$. To show our claim, it is enough to prove that 
$$\frac{|x-y|+d(x)}{\sqrt{d(x)d(y)}}\le \left(1+\frac{|x-y|}{d(x)}\right)\left(1+\frac{|x-y|}{d(y)}\right),
$$
or, equivalently,
$$d(x)d(y)\le (d(y)+|x-y|)^2
$$
which is true by the triangle inequality. The proof is complete.
\end{proof}

\begin{theorem}\label{thm-rho-u}
For all $x,y\in \Bn$ we have
$$\frac{1}{2}\rho_{\Bn}(x,y)\le u_{\Bn}(x,y)\le 4\rho_{\Bn}(x,y).
$$
\end{theorem}
\begin{proof}
From \cite[p.~151]{AVV} and from \cite[p.~29]{Vuo88} we have respectively the following two inequalities: 
\begin{equation}\label{jrho}
\tilde{j}_{\Bn}(x,y)\le \rho_{\Bn}(x,y)\le 2\tilde{j}_{\Bn}(x,y) 
\mbox{ and } \frac{1}{2}\tilde{j}_D(x,y)
\le j_D(x,y)\le \tilde{j}_D(x,y).
\end{equation}
Now the proof of our theorem follows from Lemma~\ref{lem-u-j}.
\end{proof}

Observe that for the choice of points $x,y\in \Bn$ with $y=-x$,

\[u_{\Bn}(x,-x)=2\log \left(\frac{1+|x|}{1-|x|}\right)=2 \rho_{\Bn}(0,x)=\rho_{\Bn}(x,-x).
\]

This observation leads to the following conjecture.
\begin{conjecture}
For $x,y\in \Bn$ we have $\rho_{\Bn}(x,y)\le u_{\Bn}(x,y)\le 2\rho_{\Bn}(x,y)$.
\end{conjecture}
As a consequence of Theorem~\ref{thm-rho-u} we have the following inclusion relation.
\begin{corollary}
Let $x\in \Bn$ and $t>0$. Then
$$B_{\rho_{\Bn}}(x,r)\subseteq B_{u_{\Bn}}(x,t)\subseteq B_{\rho_{\Bn}}(x,R),
$$
where $r=t/4$ and $R=2t$.
\end{corollary}
\begin{proof}
The proof follows directly from Theorem~\ref{thm-rho-u}.
\end{proof}

Next theorem shows that the $u_{\Bn}$-metric and the $\rho_{\Bn}$-metric are satisfying
the quasi-isometry property.
\begin{theorem}
For all $x,y\in \Bn$, we have 
$$
\rho_{\mathbb{B}^n}(x,y)-2\log 2\le u_{\mathbb{B}^n}(x,y)\le 2\rho_{\mathbb{B}^n}(x,y)+2\log 2.
$$
\end{theorem}

\begin{proof}
The right hand side inequality easily follows from \eqref{jrho} and \cite[Theorem~3.1]{Ibr11}.
For the left hand side inequality we assume that $x,y\in \Bn$ with $|x|\le |y|$. 
It is clear that $(1-|x|)(1-|y|)\le (1-|x|^2)(1-|y|^2)$. Now with the help of the formula \eqref{def-rho} we have
\begin{eqnarray*}
u_{\Bn} (x,y) &=& 2\log \left(\frac{|x-y|+1-|x|}{\sqrt{(1-|x|)(1-|y|)}}\right)
\ge  2\log \left(1+\frac{|x-y|}{\sqrt{(1-|x|^2)(1-|y|^2)}}\right)\\
&=& 2\log \left(1+\sinh\left(\frac{\rho_{\Bn}(x,y)}{2}\right)\right)
\ge \rho_{\mathbb{B}^n}(x,y)-2\log 2,\\
\end{eqnarray*} 
where the first inequality follows from the fact that $(1-|x|)/(1-|y|)\ge 1$ and the second inequality
follows from the fact that $1+\sinh(r)\ge e^r/2, r\ge 0$. The proof is complete.
\end{proof}

Now, we compare the $u_{\mathbb{H}^n}$-metric with the $\rho_{\mathbb{H}^{n}}$-metric. Note that for $x,y\in \mathbb{H}^{n}$, the $\rho_{\mathbb{H}^{n}}$-metric can be computed by the formula (see \cite[p.~35]{Bea95})
\begin{equation}\label{rho-hn}
2\sinh\left(\frac{\rho_{\mathbb{H}^{n}}(x,y)}{2}\right)=\frac{|x-y|}{\sqrt{x_n y_n}}.
\end{equation}

\begin{theorem}
For $x,y\in \mathbb{H}^n$ we have
$$\rho_{\mathbb{H}^n}(x,y)\le u_{\mathbb{H}^n}(x,y).
$$
The inequality is sharp.
\end{theorem}
\begin{proof}
Suppose that $x,y\in \mathbb{H}^n$. Without loss of generality we assume that $x_n\ge y_n$. Now,
\begin{eqnarray*}
u_{\mathbb{H}^n}=2\log\left(\frac{|x-y|+x_n}{\sqrt{x_n y_n}}\right) \ge 2\log\left(2\sinh\left(\frac{\rho_{\mathbb{H}^n}(x,y)}{2}\right)+1\right)\ge \rho_{\mathbb{H}^n}(x,y), 
\end{eqnarray*}
where the first inequality follows from \eqref{rho-hn} and the hypothesis. However, the second inequality follows from the fact that
$1+2\sinh (r)=1+e^r-e^{-r}\ge e^r$ for $r\ge 0$. For sharpness, consider the points $x=te_2$ and $y=(1/t)e_2$ with $t>1$. Then
$$\rho_{\mathbb{H}^n}(te_2,(1/t)e_2)=2\sinh^{-1} \left(\frac{t^2-1}{2t}\right)=2\log t \mbox{ and }
u_{\mathbb{H}^n}(te_2,(1/t)e_2)=2\log\left(\frac{2t^2-1}{t}\right).
$$
Now taking the limits as $t\to \infty$ we get
$$\lim_{t\to \infty} \frac{\rho_{\mathbb{H}^n}(te_2,(1/t)e_2)}{u_{\mathbb{H}^n}(te_2,(1/t)e_2)}=\lim_{t\to
\infty} \frac{\log t}{\log\left(\ds\frac{2t^2-1}{t}\right)}=1.
$$
Hence completing the proof.
\end{proof}

\section{Comparison of the $u_D$-metric and the $\tilde{\tau}_D$-metric}\label{u-tildetau}
We begin this section with the proof of the comparisons stated in Table~\ref{T1}.
\begin{center}
\begin{table}[h!]
\begin{tabular}{ | c | c | c | }
\hline
& Comparison with $\tilde{\tau}_D$ & Comparison with $u_D$\\[1mm]
\hline
&&\\[-3mm]
& $\frac{1}{2}j_D\le \tilde{\tau}_D\le j_D$  & $2j_D\le u_D\le 4j_D$\\[1mm]
$j_D$ & \cite[Theorems~5.1,5.4]{Ibr16} & [Lemma~\ref{lem-u-j}]\\[1mm] 
&(Right hand side inequality is sharp)&(Left hand side inequality is sharp)\\
\hline
&&\\[-3mm]
& $\frac{1}{2}\tilde{j}_D\le \tilde{\tau}_D\le \tilde{j}_D$ & $\tilde{j}_D\le u_D\le 2\tilde{j}_D$\\[1mm]
$\tilde{j}_D$ & \cite[Lemma~4.1, Theorem~4.3]{Ibr16}& [Theorem~\ref{lem-tildej-u}]\\[1mm]
&(Both the inequalities are sharp)&(Left hand side inequality is sharp)\\
\hline
&&\\[-3mm]
&  & $2\tilde{\tau}_D\le u_D\le 4\tilde{\tau}_D$\\[1mm]
$\tilde{\tau}_D$ & - & [Theorem~\ref{tau-u-g}]\\[1mm]
&& (Both the inequalities are sharp)\\[1mm]
\hline
\end{tabular}
\vspace*{0.5cm}
\caption{Comparison of $\tilde{\tau}_D$-metric and $u_D$-metric with other hyperbolic-type metrics}\label{T1}
\end{table}
\end{center}

First, we compare the $u_D$-metric with the $\tilde{\tau}_D$-metric in arbitrary domains $D\subsetneq \Rn$. Ibragimov in \cite{Ibr16}, proved that the $\tilde{\tau}_D$-metric is Gromov-hyperbolic ($\eta$-hyperbolic) with the constant $\eta=\log 3$ by comparing this with the $\tilde{j}_D$-metric. 
Comparison of the $u_D$-metric and the $\tilde{\tau}_D$-metric leads to an improvement in the constant of Gromov hyperbolicity of the $\tilde{\tau}_D$-metric from $\log 3$ to $\log 2$. In light of \cite[Theorem~3.1]{Ibr11} and \cite[Theorem~3.7]{Ibr16},
it can easily be seen that
$$u_D(x,y)\leq 4\tilde\tau_{D}(x,y)+2\log 2.
$$
holds for $x,y\in D\subsetneq \Rn$; however, this has been improved in Theorem~\ref{tau-u-g}.

Next theorem proves the comparison of both the $\tilde{\tau}_D$-metric and $u_D$-metric in the other way, i.e., the $\tilde{\tau}_D$-metric as a lower bound to the $u_D$-metric.
\begin{theorem}\label{tildetau-u}
Let $D\subsetneq \Rn$ be any domain with $\partial D\neq \emptyset$ and $x,y\in D$. Then 
$$ 2\tilde{\tau}_D(x,y)\le u_D(x,y).
$$
Equality holds whenever $d(x)=|x-p|=|y-p|=d(y)$ for some $p\in \partial D$. Moreover, there exists no
constant $k\ge 0$ such that
$$ u_D(x,y)\le 2\tilde{\tau}_D(x,y)+k
$$
for all $x,y\in D$ unless $D$ is a once-punctured space. If $D$ is a once-punctured space, then
$$ u_D(x,y)\le 2\tilde{\tau}_D(x,y)+ 2\log 2
$$
and the inequality is sharp.
\end{theorem}
\begin{proof}
Without loss of generality we assume that $d(x)\ge d(y)$. For $x,y\in D$, the relation 
$|x-p||y-p|\ge d(x)d(y)$ clearly holds for all $p\in \partial D$. Then we have
\begin{eqnarray*}
u_D(x,y) &=& 2 \log \frac{|x-y|+d(x)}{\sqrt{d(x)d(y)}}
\ge  2\log \left(1+\frac{|x-y|}{\sqrt{d(x)d(y)}}\right)\\
&\ge &2 \log\left(1+\sup_{p\in \partial D}\frac{|x-y|}{\sqrt{|x-p||p-y|}}\right)=2 \tilde{\tau}_D(x,y),\\
\end{eqnarray*}
where the first inequality follows from the fact that $d(x)/d(y)\ge 1$. It is clear that if  $d(x)=|x-p|=|y-p|=d(y)$ for some $p\in \partial D$, then both the above inequalities turn into an equality and hence the sharpness part is proved.

To prove the second part, suppose that $D$ has more than one boundary point and $k\ge 0$
such that $u_D(x,y)\le 2\tilde{\tau}_D(x,y)+k$ for all $x,y\in D$. Since $u_D$ is $\delta$-hyperbolic
in $D$ and  $2\tilde{\tau}_D(x,y)\le u_D(x,y)\le 2\tilde{\tau}_D(x,y)+k$, we conclude that $\tilde{\tau}_D$
is $\delta$-hyperbolic in $D$, contradicting \cite[Remark~4.4]{Ibr16}.

The translation invariance of the $\tilde{\tau}_D$-metric and the $u_D$-metric 
allows us to take the punctured space to be $D_0$ without any loss to 
generality. Again we assume that $|x|\ge |y|$. To show the third part, it is sufficient to
show 
$$ \frac{|x-y|+|x|}{\sqrt{|x||y|}}\le 2 \left(1+\frac{|x-y|}{\sqrt{|x||y|}}\right),
$$
or, equivalently, 
$$ \frac{|x|-|x-y|}{\sqrt{|x||y|}}\le 2.
$$
The hypothesis $|x|\ge |y|$ along with the triangle inequality yields
$$ \frac{|x|-|x-y|}{\sqrt{|x||y|}}\le \frac{|y|}{\sqrt{|x||y|}}\le 2.
$$
To prove the sharpness, let $y=e_1$ and $x=te_1, t>1$. Then
$$\lim_{t\to \infty} u_D(x,y)-2\tilde{\tau}_D(x,y)=\lim_{t\to \infty} 2\log \frac{2t-1}{t+\sqrt{t}-1}=2\log 2.
$$
Hence the proof is complete.
\end{proof}

The following inclusion relation holds true between the $u_D$-metric ball and the $\tilde{\tau}_D$-metric ball.
\begin{corollary}
Let $D\subsetneq \Rn$ be any arbitrary domain and $x\in D$. Then
$$B_{u_D}(x,r)\subseteq B_{\tilde{\tau}_D}(x,t),
$$
where $r=2t$ and the inclusion is sharp.
\end{corollary}
\begin{proof}
Suppose that $y\in B_{u_D}(x,2t)$. Then $u_D(x,y)<2t$. Now it follows from Theorem~\ref{tildetau-u} that $\tilde{\tau}_D(x,y)<t$ and hence the proof is complete. For the sharpness, consider the domain $D=\Rn\setminus\{0\}$ and $x\in D$. Choose the point $y=\partial B_{u_D}(x,2t)\cap \partial B(0,|x|)$. Now,
$u_D(x,y)=2t$ implies $\tilde{\tau}_D(x,y)=t$. This proves our corollary.
\end{proof}
Next, we aim to prove the other way of comparison. That is, to find a constant $k$ such that 
$u_D\le k \tilde{\tau}_D$. First, we prove this result in once-punctured spaces and then we extend this to arbitrary proper subdomains of $\Rn$. Next result shows that in punctured spaces the constant $k=4$.

\begin{lemma}\label{u-tau-d0}
Let $x,y\in D_0$. Then 
$$u_{D_0}(x,y)\le 4\tilde{\tau}_{D_0}(x,y).
$$
\end{lemma}
\begin{proof}
Without loss of generality, we assume that $|x|\ge |y|$. To prove the required inequality, it suffices to show that
$$\frac{|x-y|+|x|}{\sqrt{|x||y|}}\le \left(1+\frac{|x-y|}{\sqrt{|x||y|}}\right)^2
$$
or, equivalently,
$$ \frac{|x|-|x-y|}{\sqrt{|x||y|}}\le 1+ \frac{|x-y|^2}{|x||y|}.
$$
This holds true, because
$$\frac{|x|-|x-y|}{\sqrt{|x||y|}}\le \frac{|y|}{\sqrt{|x||y|}}\le 1 \le 1+ \frac{|x-y|^2}{|x||y|},
$$
completing the proof of our lemma.
\end{proof}

We now prove that the conclusion of Lemma~\ref{u-tau-d0} still holds if we replace the once-punctured space by twice-punctured spaces.

\begin{lemma}\label{u-tau}
Let $x,y\in D_{p,q}$. Then 
$$u_{D_{p,q}}(x,y)\le 4\tilde{\tau}_{D_{p,q}}(x,y).
$$
\end{lemma}
\begin{proof}
Suppose that $x,y\in D_{p,q}$. If $d(x)=|x-p|$ and $d(y)=|y-p|$ (or $d(x)=|x-q|$ and $d(y)=|y-q|$), then 
the proof follows from Theorem~\ref{u-tau-d0}. Hence, without loss of generality, we assume
$d(x)=|x-p|,d(y)=|y-q|$, and $|x-p|\ge |y-q|$. 
Note that $\tilde{\tau}_{D_{p,q}}(x,y)=\tilde{\tau}_{D_{p}}(x,y)\vee \tilde{\tau}_{D_{q}}(x,y)$.
Hence, to prove our claim, it is enough to establish the inequality $u_{D_{p,q}}(x,y)\le 4 \tilde{\tau}_{D_q}(x,y)$. That is to prove the inequality
$$\frac{|x-y|+|x-p|}{\sqrt{|x-p||y-q|}}\le \left(1+\frac{|x-y|}{\sqrt{|x-q||y-q|}}\right)^2.
$$

Let $|x-y|=a|y-q|$, where $a>0$. From the assumption we know that 
\be\label{proof-1}|x-p|\leq |x-q|\leq |x-y|+|y-q|.\ee
Now 
\beq\label{proof-2}\frac{|x-y|+|x-p|}{\sqrt{|x-p||y-q|}}&=&\frac{|x-y|}{\sqrt{|x-p||y-q|}}+\frac{|x-p|}{\sqrt{|x-p||y-q|}}\\ \nonumber&\leq&
\frac{|x-y|}{|y-q|}+\sqrt{\frac{|x-p|}{|y-q|}}\\ \nonumber&\leq& \frac{|x-y|}{|y-q|}+\sqrt{\frac{|x-y|}{|y-q|}+1}\\ \nonumber&\leq&
a+\sqrt{a+1}.\eeq

We obtain from (\ref{proof-1}) that \be\label{proof-3}\frac{2|x-y|}{\sqrt{|x-q||y-q|}}\geq \frac{2|x-y|}{\sqrt{1+a}|y-q|}=\frac{2a}{\sqrt{1+a}}.\ee

Next we divide the proof into two cases.

\bca $a<1$.\eca

It follows from 
(\ref{proof-2}) and (\ref{proof-3}) that 
\beq 
\bigg(1+\frac{|x-y|}{\sqrt{|x-q||y-q|}}\bigg)^2-\frac{|x-y|+|x-p|}{\sqrt{|x-p||y-q|}}
\nonumber &> & \frac{2|x-y|}{\sqrt{|x-q||y-q|}}+1-\frac{|x-y|+|x-p|}{\sqrt{|x-p||y-q|}}\\
\nonumber&\geq&\frac{2a}{\sqrt{1+a}}+1-a-\sqrt{a+1}>0
\eeq

\bca $a\ge 1$.\eca

From (\ref{proof-1}) we have 
\be \label{proof-4}
\frac{|x-y|^2}{|x-q||y-q|}\geq \frac{|x-y|^2}{(|x-y|+|y-q|)|y-q|}=\frac{a^2}{1+a}.
\ee
Then it follows from (\ref{proof-2}), (\ref{proof-3}) and \eqref{proof-4} that
\beq \bigg(1+\frac{|x-y|}{\sqrt{|x-q||y-q|}}\bigg)^2-\frac{|x-y|+|x-p|}{\sqrt{|x-p||y-q|}}\nonumber &\geq & \frac{2a}{\sqrt{1+a}}+\frac{a^2}{1+a}+1-a-\sqrt{a+1}\\ \nonumber
&\ge & 0.
\eeq
This completes the proof of our lemma.
\end{proof}

Combining Lemma~\ref{tildetau-u} and Lemma~\ref{u-tau} we obtain

\begin{theorem}\label{tau-u-g}
Let $x,y\in D\subsetneq \Rn$. Then
$$2 \tilde{\tau}_D(x,y) \le u_D(x,y)\le 4\tilde{\tau}_D(x,y).
$$
Both the inequalities are sharp.
\end{theorem}
\begin{proof}
The first inequality is proved in Theorem~\ref{tildetau-u}. Now, we prove the second inequality.
Suppose that $p,q\in \partial D$ such that $d(x)=|x-p|$ and $d(y)=|y-q|$. 
Clearly, $D\subset D_{p,q}$ and $u_D(x,y)=u_{D_{p,q}}(x,y)$. Now,
$$u_D(x,y)=u_{D_{p,q}}(x,y)\le 4 \tilde{\tau}_{D_{p,q}}(x,y)\le 4\tilde{\tau}_{D}(x,y),
$$ 
where the first inequality follows from Lemma~\ref{u-tau} and the second inequality follows from
the monotone property of $\tilde{\tau}_D$ \cite[p.~2]{Ibr16}.

The sharpness of the first inequality is given in Theorem~\ref{tildetau-u}. For the sharpness of the second inequality we consider the unit ball $\Bn$. Choose the points $x$ and $y$ such that $y=-x$. Now,
$$u_{\Bn}(x,-x)=2\log\left(\frac{1+|x|}{1-|x|}\right) \mbox{ and } 
\tilde{\tau}_{\Bn}(x,-x)=\log\left(1+\frac{2|x|}{\sqrt{1-|x|^2}}\right).
$$
It follows that
$$\lim_{|x|\to 1} \frac{u_{\Bn}(x,-x)}{4\tilde{\tau}_{\Bn}(x,-x)}=\lim_{|x|\to 1} \frac{(2|x|+\sqrt{1-|x|^2})^2}{(1+|x|)^2}=1.
$$
Hence the proof is complete.
\end{proof}

An immediate consequence of Theorem~\ref{tau-u-g} is the following inclusion relation.

\begin{corollary}\label{u-tau-inclusion}
Let $x\in D\subsetneq \Rn$ and $t>0$. Then we have
$$B_{u_D}(x,r)\subseteq B_{\tilde{\tau}_D}(x,t)\subseteq B_{u_D}(x,R),
$$
where $r=2t$ and $R=4t$. The radii $r$ and $R$ are the best possible.
\end{corollary}
\begin{proof}
Let $y\in B_{u_D}(x,r)$, $r=2t$ and $R=4t$. Then by Theorem~\ref{tau-u-g} we have $\tilde{\tau}_D(x,y)<t$. So, $B_{u_D}(x,r)\subseteq B_{\tilde{\tau}_D}(x,t)$. Conversely, if $y\in B_{\tilde{\tau}_D}(x,t)$, then also by Theorem~\ref{tau-u-g} we have $y\in B_{u_D}(x,R)$. So, $B_{\tilde{\tau}_D}(x,t)\subseteq B_{u_D}(x,R)$ and hence the inclusion follows. Next, we need to prove the sharpness part.

First we consider the domain $D=D_0$ and let $x\in D_0$. Now choose $y\in B(0,|x|)\cap \partial B_{\tilde{\tau}_{D_0}}(x,t)$. Then
$$\tilde{\tau}_{D_0}(x,y)=t=\log\left(1+\frac{|x-y|}{|x|}\right)=\frac{u_{D_0}(x,y)}{2},
$$
which proves the sharpness of the first inclusion. 
Secondly, consider $D=\Bn$ and let $x\in \Bn$ be arbitrary. 
Choose $y\in \Bn$ such that $x$ and $y$ lie on a diameter of $\Bn$ with $0$ lying 
in-between and $|y|\le |x|$. Now, 
$$u_{\Bn}(x,y)=2\log\left(\frac{|x-y+1-|y||}{\sqrt{(1-|x|)(1-|y|)}}\right) \mbox{ and }
\tilde{\tau}_{\Bn}(x,y)=\log\left(1+\frac{|x-y|}{\sqrt{(1-|x|)(1+|y|)}}\right).
$$
It follows that
\begin{eqnarray*}
\lim_{x\to e_1} \frac{u_{\Bn}(x,y)}{\tilde{\tau}_{\Bn}(x,y)} &=&
\lim_{x\to e_1} \frac{(|x-y|+1-|y|)(1-|x|)(1+|y|)}{\sqrt{(1-|x|)(1-|y|)}(|x-y|+\sqrt{(1-|x|)(1+|y|)})^2}\\
&=&\left\{
\begin{array}{ll}
1, & \mbox{ if } y=-x,\\
0, & \mbox{otherwise}.\\
\end{array}\right.\\
\end{eqnarray*}
Hence we conclude that 
for each $x\in \Bn$ with $|x|\to 1$ and $t>0$, there exist $y=-x$ such that $y\in \partial B_{\tilde{\tau}_D}(x,t)$ and $u_D(x,y)=4t$. This proves the sharpness of the second inclusion
relation and hence the proof is complete.
\end{proof}

Recall that 
\begin{equation}\label{j-tau}
\frac{1}{2} \tilde{j}_D(x,y)\le \tilde{\tau}_D(x,y)\le \tilde{j}_D(x,y)
\end{equation}
holds true for 
$D\subsetneq \Rn$ (see \cite[Theorem~4.2, 4.3]{Ibr16}). Both the inequalities are sharp. 
The proof of the sharpness part of the left hand side inequality is done by the method of contradiction in \cite{Ibr16}. 
Here we give a precise example to prove the sharpness part of the left hand side inequality.

Consider the unit ball $\Bn$ and $x,y\in \Bn$ with $y=-x$. Now we see that
$$\tilde{\tau}_{\Bn}(x,-x)=\log\left(1+\frac{2|x|}{\sqrt{1-|x|^2}}\right) \mbox{ and } \tilde{j}_{\Bn}(x,-x)=\log\left(
1+\frac{2|x|}{1-|x|}\right).
$$
It follows that
\begin{equation}\label{j-tau-eqn}
\lim_{|x|\to 1}\frac{\tilde{j}_{\Bn}(x,-x)}{2\tilde{\tau}_{\Bn}(x,-x)}=\lim_{|x|\to 1} \frac{2|x|+\sqrt{1-|x|^2}}{2}=1.
\end{equation}

Now, we establish inclusion relation between the $\tilde{j}_D$-metric and the $\tilde{\tau}_D$-metric balls.
\begin{theorem}
Let $D\subsetneq \Rn$ and $x\in D$ and $t>0$. Then the following inclusion property holds true:
$$B_{\tilde{j}_D}(x,r)\subseteq B_{\tilde{\tau}_D}(x,t)\subseteq B_{\tilde{j}_D}(x,R).
$$
Here $r=t$ and $R=2t$. The radii $r$ and $R$ are the best possible.
\end{theorem}
\begin{proof}
The proof follows from (\ref{j-tau}). To show that the radius $r$ is the best possible, consider the domain 
$D=D_0=\Rn\setminus\{0\}$ and $x\in D$. Choose $y\in \partial B(0,|x|)\cap \partial B_{\tilde{\tau}_D}(x,t)$.
Now clearly, $\tilde{j}_D(x,y)=\tilde{\tau}_D(x,y)=t$. 

To show $R$ is the best possible, consider the domain
$D=\Bn$. With the similar argument given in the proof of Corollary~\ref{u-tau-inclusion}, for the second 
inclusion property, we can show that for each $x\in \Bn$ with $|x|\to 1$ and $t>0$, there exist $y(=-x)$ such that $y\in \partial B_{\tilde{\tau}_{\Bn}}(x,t)$ and $\tilde{j}_{\Bn}(x,y)=2t$. This completes the proof of our theorem.
\end{proof}
By Theorem~\ref{tau-u-g} and \eqref{j-tau} we have 
$$\tilde{j}_D(x,y)\le 2\tilde{\tau}_D(x,y)\le u_D(x,y)
$$
and also
$$u_D(x,y)\le 4\tilde{\tau}_D(x,y)\le 4\tilde{j}_D(x,y).
$$
Hence we have the following relationship between the $\tilde{j}_D$-metric and the $u_D$-metric.

\begin{theorem}\label{lem-tildej-u}
For $D\subsetneq \Rn$ we have
$$\tilde{j}_D(x,y)\le u_D(x,y)\le 4 \tilde{j}_D(x,y).
$$
The first inequality is sharp.
\end{theorem}
\begin{proof}
For the sharpness part, consider the domain $D=\Rn\setminus\{-e_1,e_1\}$. Choose $x=0$ and $y=te_2, t>1$. Then $\tilde{j}_D(0,te_2)=\log(1+t)$ and
$$ u_D(0,te_2)=2\log\left(\frac{t+\sqrt{1+t^2}}{(1+t^2)^{1/4}}\right)=\log\left(\frac{1+2t^2+2t\sqrt{1+t^2}}{\sqrt{1+t^2}}\right).
$$
Now we see that
\begin{eqnarray*}
\lim_{t\to \infty} \frac{\tilde{j}_D(0,te_2)}{u_D(0,te_2)} &=&\lim_{t\to \infty} \frac{\log(1+t)}{\log\left(\ds\frac{1+2t^2+2t\sqrt{1+t^2}}{\sqrt{1+t^2}}\right)}\\
&=& \lim_{t\to \infty} \ds \frac{(1+t^2)(1+2t^2+2t\sqrt{1+t^2})}{(1+t)(4t(1+t^2)+2(1+t^2)^{3/2}-t-2t^3)}=1.
\end{eqnarray*}
This completes the proof of our theorem.
\end{proof}

\begin{remark}
The constant $4$ in the right hand side inequality of Lemma~\ref{lem-tildej-u} can't be replaced by $2$ due to the fact that 
$$u_D(x,y)\le 2\tilde{j}_D(x,y)\iff |x-y|^2\ge d(x)d(y)
$$
for every $x,y\in D$, which is not true in general.
\end{remark}

As an immediate consequence of Theorem~\ref{lem-tildej-u} we have the following inclusion relation.
\begin{corollary}
Let $D\subsetneq \Rn$, $x\in D$, and $t>0$. Then
$$B_{u_D}(x,r)\subseteq B_{\tilde{j}_D}(x,t)\subset B_{u_D}(x,R),
$$
where $r=t$ and $R=4t$. The radius $r$ is best possible.
\end{corollary}
\begin{proof}
Proof follows from Theorem~\ref{lem-tildej-u}.
\end{proof}

The next result shows that the $\tilde{\tau}_D$-metric balls can be written as the intersection of $\tilde{\tau}$-metric balls in punctured spaces over the boundary points of $D$.
\begin{proposition}
Let $D\subsetneq \Rn$, $x\in D$, and $r>0$. Then
$$B_{\tilde{\tau}_D}(x,r)=\cap_{p\in \partial D} B_{\tilde{\tau}_{\Rn\setminus\{p\}}}(x,r).
$$
\end{proposition}
\begin{proof}
Suppose that $y\in \cap_{p\in \partial D} B_{\tilde{\tau}_{\Rn\setminus\{p\}}}(x,r)$. Then $\tilde{\tau}_{\Rn\setminus\{p\}}(x,y)<r$ for all $p\in \partial D$. In particular, 
$$\tilde{\tau}_D(x,y)=\sup_{p\in \partial D} \tilde{\tau}_{\Rn\setminus\{p\}}(x,y)<r.
$$
So, $\cap_{p\in \partial D} B_{\tilde{\tau}_{\Rn\setminus\{p\}}}(x,r)\subseteq B_{\tilde{\tau}_D}(x,r)$. 
Conversely, suppose that $y\in B_{\tilde{\tau}_D}(x,r)$ and let $p\in \partial D$. Then
$$B_{\tilde{\tau}_D}(x,r)\subseteq B_{\tilde{\tau}_{\Rn\setminus\{p\}}}(x,r)
$$ 
by the monotone property of the $\tilde{\tau}_D$-metric.
Hence, $ B_{\tilde{\tau}_D}(x,r) \subseteq \cap_{p\in \partial D} B_{\tilde{\tau}_{\Rn\setminus\{p\}}}(x,r)$
and the proof is complete.
\end{proof}

\section{Comparison with other related metrics}\label{u-tildetau-other}
In this section, we consider the Cassinian \cite{Ibr09}, the Seittenranta \cite{Sei99}, the triangular ratio \cite{CHKV15} and the half-Apollonian  \cite{HL04} metrics, and compare them with the $\tilde{\tau}_D$-metric and the $u_D$-metric.
Main results of this section are stated in Table~\ref{T2}.
\begin{center}
\begin{table}[H]
\begin{tabular}{ | c | c | c | }
\hline
& Comparison with $\tilde{\tau}_D$ & Comparison with $u_D$\\[1mm]
\hline
&&\\[-3mm]
& $\cfrac{1}{4}\delta_D(x,y)\le \tilde{\tau}_D(x,y)\le \delta_D(x,y)$  & $\cfrac{\delta_D}{2}\le u_D\le 4j_D$\\[1mm]
$\delta_D$ & [Theorem~\ref{delta-tau}] & [Corollary~\ref{delta-u}]\\[1mm] 
&(Right hand side inequality is sharp)&\\[1mm]
\hline
&&\\[-3mm]
$c_D$ & $c_D(x,y)\le \cfrac{e^{4\tilde{\tau}_D(x,y)}-1}{d(y)}$ & $c_D(x,y)\le \cfrac{e^{2u_D(x,y)}-1}{d(y)}$\\[1mm] 
 & [Corollary~\ref{cor-cD-tau-u}] & [Corollary~\ref{cor-cD-tau-u}]\\[1mm]
\hline
&&\\[-3mm]
 & $(\log 3)s_D\le \tilde{\tau}_D$ & $(\log 9)s_D\le  u_D$\\[1mm]
$s_D$ & [Theorem~\ref{tau-sD}] & [Corollary~\ref{cor-sD-u}]\\[1mm]
&The inequality is sharp&\\
\hline 
&&\\[-3mm]
 & $\frac{1}{2}\eta_D\le \tilde{\tau}_D\le \log(2+e^{\eta_D})$ & $\eta_D\le u_D\le 4\log(2+e^{\eta_D})$\\[1mm]
$\eta_D$ & \cite{IS} & [Lemma~\ref{eta-u}]\\[1mm]
&(Both inequalities are sharp)& \\[1mm]
\hline

\end{tabular}
\vspace*{0.5cm}
\caption{Comparisons with other hyperbolic-type metrics}\label{T2}
\end{table}
\end{center}

We begin by comparing the Cassinian metric with the Seittenranta metric. 
The {\it Cassinian metric}, $c_D$, of the domain $D\subsetneq \Rn$ is defined as
$$
c_D (x,y)=\sup_{p\in \partial D} \frac{|x-y|}{|x-p||p-y|} .
$$
This metric was first introduced and studied in \cite{Ibr09} and subsequently studied in \cite{IMSZ,HKVZ,KMS}.
Geometrically, the $c_D$-metric can be defined by taking the maximal Cassinian oval in $D$ with foci at
$x$ and $y$ (see, \cite{Ibr09}).
Clearly, the supremum in the definition is attained at some point $p \in \partial D$.

The {\it{Seittenranta metric}}, $\delta_D$, introduced in \cite{Sei99}, is defined by
$$\delta_D(x,y)=\log(1+m_D(x,y))
$$
where 
$$m_D(x,y)=\sup_{a,b\in \partial D} \frac{|x-y||a-b|}{|x-a||y-b|}.
$$
Note that the quantity $m_D(x,y)$ does not define a metric. The Seittenranta metric is M\"obius invariant and coincides with the hyperbolic metric of the unit ball $\Bn$. 

The $c_D$-metric and the $\delta_D$-metric are exponentially related, which is stated in the following
theorem.

\begin{theorem}\label{c-delta}
Let $D\subsetneq \overline{\Rn}$ be any domain. Then
$$c_D(x,y)\le \frac{e^{\delta_D(x,y)}-1}{d(y)}.
$$
The inequality is sharp. 
\end{theorem}
\begin{proof}
Let $p\in \partial D$ such that
$$c_D(x,y)=\frac{|x-y|}{|x-p||p-y|}.
$$
Choose $q\in \partial D$ such that $|p-q|\ge |y-q|$. Now,
\begin{eqnarray*}
c_D(x,y) &=& \frac{|x-y|}{|x-p||p-y|}\\
&=& \frac{|x-y||p-q|}{|x-p||y-q|}.\frac{|y-q|}{|p-q||p-y|}\\
&\le & \frac{m_D(x,y)}{d(y)}.
\end{eqnarray*}
Hence we get
$$\delta_D(x,y)=\log(1+m_D(x,y))\ge \log(1+d(y).c_D(x,y))
$$
and the proof is complete. For the sharpness, we consider the punctured space $D_p$. Let $x,y\in D_p$ with
$|x-p|\le |y-p|$. It is clear that $\delta_{D_p}(x,y)=\tilde{j}_{D_p}(x,y)=\log(1+|x-y|/|x-p|)$ and hence the sharpness
follows.
\end{proof}
An immediate corollary to Theorem~\ref{c-delta} is the following inclusion relation.
\begin{corollary}
Let $D\subsetneq \Rn$, $x\in D$, and $t>0$. Then
$$B_{\delta_D}(x,t)\subseteq B_{c_D}(x,R),
$$
where $R=(e^t-1)/d(y)$. The inclusion is sharp.
\end{corollary}
\begin{proof}
If $\delta_D(x,y)<y$, then by Theorem~\ref{c-delta}, we have $c_D(x,y)<(e^t-1)/d(y)$. For the sharpness, 
choose a point $y\in \partial B_{\delta_D}(x,t)\cap L$, where $L$ is the line passing 
through $0$ and $x$ with $|x|<|y|$. Then
$$\delta_D(x,y)=t=\log\left(1+\frac{|x-y|}{|x|}\right).
$$
Now, 
$$c_D(x,y)=\frac{|x-y|}{|x||y|}=\frac{1}{|y|}(e^t-1)
$$
and hence the proof is complete.
\end{proof}

Again the $\delta_D$-metric is bilipschitz equivalent to the $\tilde{j}_D$-metric. Indeed, we have
\begin{equation}\label{tildej-delta}
\tilde{j}_D\le \delta_D\le 2\tilde{j}_D,
\end{equation}
see, for instance \cite[p.~525]{Sei99}. Hence \eqref{j-tau} along with \eqref{tildej-delta} yield the
following inequality between the $\delta_D$-metric and the $\tilde{\tau}_D$-metric.
\begin{theorem}\label{delta-tau}
Let $x,y\in D\subsetneq \Rn$. Then the following holds true:
$$\frac{1}{4}\delta_D(x,y)\le \tilde{\tau}_D(x,y)\le \delta_D(x,y).
$$
The second inequality is sharp.
\end{theorem}
\begin{proof}
The proof of the inequality follows directly from \eqref{j-tau} and \eqref{tildej-delta}.
For the sharpness of the second inequality, consider the domain $D_0$ and choose $x,y\in D_0$ with $y=-x$.
Then $\tilde{\tau}_{D_0}(x,-x)=\log 3=\delta_{D_0}(x,-x)$.
\end{proof}
The following inclusion relation holds true.
\begin{corollary}
Let $x\in D\subsetneq \Rn$ and $t>0$. Then
$$B_{\delta_D}(x,r)\subseteq B_{\tilde{\tau}_D}(x,t)\subseteq B_{\delta_D}(x,R),
$$
where $r=t$ and $R=4t$.
\end{corollary}
\begin{proof}
Proof follows from Theorem~\ref{delta-tau}.
\end{proof}

Theorem~\ref{tau-u-g} and Theorem~\ref{delta-tau} together yield the following:
\begin{corollary}\label{delta-u}
Let $x,y\in D\subsetneq \Rn$. Then we have
$$\frac{1}{2}\delta_D(x,y)\le u_D(x,y)\le 4\delta_D(x,y).
$$
\end{corollary}
Corollary~\ref{delta-u} leads to the following inclusion relation.
\begin{corollary}
Let $x\in D\subsetneq \Rn$ and $t>0$. Then
$$B_{\delta_D}(x,r)\subseteq B_{u_D}(x,t)\subseteq B_{\delta_D}(x,R),
$$
where $r=t/4$ and $R=2t$.
\end{corollary}
Hence, as a consequence to Theorem~\ref{c-delta}, we have
\begin{corollary}\label{cor-cD-tau-u}
Let $x,y\in D\subsetneq \Rn$. Then we have
$$c_D(x,y)\le \frac{e^{4\tilde{\tau}_D(x,y)}-1}{d(y)} \mbox{ and } c_D(x,y)\le \frac{e^{2u_D(x,y)}-1}{d(y)}.
$$
\end{corollary}
\begin{proof}
The first inequality follows from Theorem~\ref{c-delta} and Lemma~\ref{delta-tau}, whereas the second inequality
follows from Theorem~\ref{c-delta} and Corollary~\ref{delta-u}.
\end{proof}

Now, we compare the $\tilde{\tau}_D$-metric with 
the triangular ratio metric, $s_D$, defined in a proper subdomain $D$ of $\Rn$ by
$$s_D(x,y)=\sup_{p\in \partial D} \frac{|x-y|}{|x-p|+|p-y|}, \quad x,y\in D.
$$
Geometrically, the triangular ratio metric can be viewed by taking the maximal ellipse in $D$ with foci at $x$ and $y$ in the similar fashion as in the geometric definition the Apollonian metric \cite{Has03}. For more details on $s_D(x,y)$ we refer \cite{CHKV15}.

\begin{theorem}\label{tau-sD}
Let $D\subsetneq \Rn$ and $x,y\in D$. Then
$$\tilde{\tau}_D(x,y)\ge (\log 3) s_D(x,y).
$$
The inequality is sharp.
\end{theorem}
\begin{proof}
From AM-GM inequality, it follows that
$$\frac{1}{\sqrt{|x-p||y-p|}}\ge \frac{2}{|x-p|+|y-p|}.
$$
Now,
\begin{eqnarray*}
\tilde{\tau}_D(x,y) &\ge & \log\left(1+\frac{|x-y|}{\sqrt{|x-p||y-p|}}\right)\\
&\ge & \log\left(1+\frac{2|x-y|}{|x-p|+|y-p|}\right)\\
&\ge & \frac{|x-y|}{|x-p|+|y-p|} \log 3
\end{eqnarray*}
holds for all $p\in \partial D$. Here the last inequality follows from the well known Bernoulli's 
inequality:
$$\log(1+ax)\ge a \log(1+x)\quad \mbox{ for } a\in (0,1), x>0.
$$
 In particular, we have $\tilde{\tau}_D(x,y)\ge (\log 3) s_D(x,y)$. To prove the sharpness, 
consider the domain $D_0$ and $x,y\in D_0$ with $y=-x$. Then 
$\tilde{\tau}_D(x,-x)=\log 3$ and $s_D(x,-x)=1$.
\end{proof}

\begin{remark}
Combining Theorem~\ref{tau-sD} and \cite[Theorem~4.3]{Ibr16}, one can obtain Theorem~3.3 of \cite{CHKV15}.
\end{remark}
Theorem~\ref{tau-sD} leads to the following inclusion relation.
\begin{corollary}
Let $x\in D\subsetneq \Rn$ and $t>0$. Then 
$$B_{\tilde{\tau}_D}(x,t)\subseteq B_{s_D}(x,R), 
$$
where $R=t/\log 3$. The inclusion is sharp.
\end{corollary}
\begin{proof}
It follows from Theorem~\ref{tau-sD} that for $\tilde{\tau}_D(x,y)<t$, 
$s_D(x,y)<t/\log 3$. Hence, $B_{\tilde{\tau}_D}(x,t)\subseteq B_{s_D}(x,R)$ with $R=t/\log 3$.

To prove the sharpness part, consider the domain $D=D_0$ and $t=\log 3$. Then we have $R=1$ and
$$\tilde{\tau}_{D_0}(x,y)=\log 3 \iff |x-y|=2\sqrt{|x||y|} 
$$
and hence
$$s_{D_0}(x,y)=\frac{2\sqrt{|x||y|}}{|x|+|y|}.
$$
To show $s_{D_0}(x,y)=1$, we need to choose points $x$ and $y$ such that $|x|=|y|$. This implies $x$ and $y$ are co-linear. i.e. $y=-x$. From the definition of the $\tilde{\tau}_D$ metric it is clear that the point
$-x$ lies on the sphere $\partial B_{\tilde{\tau}_{D_0}}(x,\log 3)$. Now, for any $x\in D_0$, choose 
$y\in \partial B_{\tilde{\tau}_{D_0}}(x,\log 3)\cap L$, where $L$ is the line passing through $0$ and $x$
with $|x|=|y|$. Then
$$\tilde{\tau}_{D_0}(x,y)=\log 3 \iff s_{D_0}(x,y)=1.
$$
Hence, the proof is complete.
\end{proof}

Another consequence of Theorem~\ref{tau-sD} leads to the following corollary.

\begin{corollary}\label{cor-sD-u}
Let $D\subsetneq \Rn$. Then for all $x,y\in D$ we have
$$s_D(x,y)\le \frac{1}{\log 9} u_D(x,y).
$$
\end{corollary}
\begin{proof}
The proof follows from Theorem~\ref{tau-u-g} and Theorem~\ref{tau-sD}.
\end{proof}

Next, we discuss the inclusion properties associated with the $\tilde{\tau}_D$-metric and the half-Apollonian metric balls. The half-Apollonian metric, $\eta_D$, of a domain $D\subsetneq \Rn$ is defined by
$$\eta_D(x,y)=\sup_{p\in \partial D} \left|\log \frac{|x-p|}{|y-p|}\right|, \quad x,y\in D.
$$
The $\eta_D$-metric was introduced by H\"ast\"o and Linden in \cite{HL04}. 
It is now appropriate to recall the following result by Seittenranta.

\begin{lemma}\cite[Theorem~3.11]{Sei99}\label{sei3.11}
Let $D\subset \overline{\Rn}$ be an open set with ${\rm card }\,\, \partial D\ge 2$. Then
$$\alpha_D(x,y)\le \delta_D(x,y)\le \log (e^{\alpha_D}+2).
$$
The inequalities give the best possible bounds for $\delta_D$ expressed in terms of
$\alpha_D$ only.
\end{lemma}
Here the quantity $\alpha_D$ represents the Apollonian metric, introduced by Beardon in \cite{Bea98}.
As a special case of Lemma~\ref{sei3.11}, the following result holds true, which is proved in \cite{IS}.

\begin{lemma}\label{tau-eta}
Let $D\subsetneq \Rn$ and $x,y\in D$. Then
$$ \frac{1}{2} \eta_D(x,y) \le \tilde{\tau}_D(x,y) \le \log(2+e^{\eta_D(x,y)}).
$$
Both the inequalities are sharp.
\end{lemma}
Lemma~\ref{tau-eta} obtains the following inclusion property.

\begin{corollary}
Let $D\subsetneq \Rn$ and $x\in D$ and $t>0$. Then the following inclusion property holds true:
$$B_{\eta_D}(x,r)\subseteq B_{\tilde{\tau}_D}(x,t)\subseteq B_{\eta_D}(x,R).
$$
Here $r=\log(e^t-2)$ and $R=2t$. The radii $r$ and $R$ are best possible.
\end{corollary}
\begin{proof}
Let $y\in B_{\tilde{\tau}_D}(x,t)$, i.e. $\tilde{\tau}_D(x,y)<t$. 
From the left hand side inequality of 
Theorem~\ref{tau-eta}, we have $\eta_D(x,y)<2t(=R)$. On the other hand, if $\eta_D(x,y)<\log(e^t-2)(=r)$, then
from the right hand side inequality of Theorem~\ref{tau-eta}, we have $\tilde{\tau}_D(x,y)<t$. With the similar argument
given in the proof for the sharpness part of second inclusion relation in Corollary~\ref{u-tau-inclusion}, we 
can show that the radius $R$ is the best possible in the punctured space $D_0$ with $t=\log 3$. 
\end{proof}

Comparison of Theorem~\ref{tau-u-g} and Lemma~\ref{tau-eta} together lead to the following relation 
between the $\eta_D$-metric and the $\tilde{\tau}_D$-metric.
\begin{lemma}\label{eta-u}
Let $D\subsetneq \Rn$ and $x,y\in D$. Then
$$\eta_D(x,y)\le u_D(x,y)\le 4\log(2+e^{\eta_D(x,y)}).
$$
\end{lemma}
Subsequently, we have the following trivial inclusion relation.
\begin{corollary}
Let $x\in D\subsetneq \Rn$ and $t>0$. Then
$$B_{\eta_D}(x,r)\subseteq B_{u_D}(x,t)\subseteq B_{\eta_D}(x,R),
$$
where $r=\log(e^{t/4}-2)$ and $R=t$.
\end{corollary} 

\bigskip

\begin{acknowledgement}

The second author would like to thank Zair Ibragimov for bringing the interesting paper \cite{Ibr11} to his attention and for 
useful discussion on this topic when the author visited him during June 2016.
The authors would also like to thank Manzi Huang for her valuable comments, specially
for nice conversation in the proof of Lemma~\ref{u-tau}. The research was partially 
supported by NBHM, DAE $($Grant No: $2/48 (12)/2016/${\rm NBHM (R.P.)/R \& D II}/$13613)$.
\end{acknowledgement}


\begin{thebibliography}{99}
\bibitem{AVV}  {G. D. Anderson, M. K. Vamanamurthy, and M. K.
Vuorinen,}
{\em Conformal Invariants, Inequalities, and Quasiconformal Maps},
John Wiley \& Sons, Inc., 1997.

\bibitem{Bea95} {A. F. Beardon},
{\em Geometry of discrete groups},
Springer-Verlag, New York, 1995.

\bibitem{Bea98}  {A. F. Beardon},
{\em The Apollonian metric of a domain in $\mathbb{R}^n$}. In:
P. Duren, J. Heinonen, B. Osgood, and B. Palka (Eds.)
\textit{Quasiconformal mappings and analysis} (Ann Arbor, MI, 1995),
pp. 91--108.
Springer-Verlag, New York, 1998.

\bibitem{BI08}{M. Borovikova and Z. Ibragimov,}
{\em Convex bodies of constant width and the Apollonian metric},
Bull. Malays. Math. Sci. Soc., {\bf 31} (2) (2008), 117--128.

\bibitem{BS00} {Bonk M., Schramm O.},
{\em Embeddings in Gromov hyperbolic spaces},
Geom. Funct. Anal., {\bf 10} (2) (2000), 266--306.

\bibitem{CHKV15} { J. Chen, P. Hariri, R. Kl\'en} and {M. Vuorinen},
{\em Lipschitz conditions, Triangular ratio metric and Quasiconformal maps},
Ann. Acad. Sci. Fenn. Math., {\bf 40} (2015), 683--709.


\bibitem{GO79} {F. W. Gehring and B. G. Osgood,}
{\em Uniform domains and the quasihyperbolic metric},
{J. Analyse Math.,} {\bf 36} (1979), 50--74.

\bibitem{GP76} {F. W. Gehring and B. P. Palka,}
{\em Quasiconformally homogeneous domains},
{J. Analyse Math.,} {\bf 30} (1976), 172--199.

\bibitem{Gro87} {Gromov M.}, 
{\em Hyperbolic groups}, Essays in Group Theory, 75--263,
Math. Sci. Res. Inst. Publ., {\bf 8}, Springer, New York, 1987.


\bibitem{HKVZ} {P. Hariri, R. Kl\'en, M. Vuorinen, and X.H- Zhang},
{\em Some remarks on the Cassinian metric},
Publ. Math. Debrecen., {\bf 90} (3-4) (2017), 269--285.

\bibitem{Has02} {P. H\"ast\"o,}
{\em A new weighted metric: the relative metric. I},
{J. Math. Anal. Appl.}, {\bf 274} (1) (2002), 38--58.

\bibitem{Has03} {P. H\"ast\"o,}
{\em The Apollonian metric: uniformity and quasiconvexity,}
{Ann. Acad. Sci. Fenn. Math.}, {\bf 28} (2) (2003), 385--414.

\bibitem{Has05} {P. H\"ast\"o,}
{\em Gromov hyperbolicity of the $j_G$ and $\tilde{j}_G$ metrics},
{Proc. Amer. Math. Soc.,} {\bf 134} (4) (2005), 1137--1142.

\bibitem{HL04} {P. H\"ast\"o and H. Linden},
{\em Isometries of the half-apollonian metric},
Complex Var. Theory Appl., {\bf 49} (2004), 405--415.


\bibitem{Ibr03}{Z. Ibragimov,}
{\em On the Apollonian metric of domains in $\overline{\mathbb{R}^n}$}.
Complex Var. Theory Appl., {\bf 48} (10) (2003), 837--855.


\bibitem{Ibr09}{Z. Ibragimov,}
{\em The Cassinian metric of a domain in $\bar{\mathbb{R}}^n$},
Uzbek. Mat. Zh., (2009), {1}, 53--67.

\bibitem{Ibr11}{Z. Ibragimov,}
{\em Hyperbolizing metric spaces,}
Proc. Amer. Math. Soc., {\bf 139} (12) (2011), 4401--4407.

\bibitem{Ibr11-I} {Ibragimov Z.},
{\em Hyperbolizing hyperspaces,}
Michigan Math. J., {\bf 60} (2011), 215--239.

\bibitem{Ibr16}{Z. Ibragimov,}
{\em A scale-invariant Cassinian metric},
J. Anal., {\bf 24} (1) (2016), 111-129.

\bibitem{IS}{Z. Ibragimov and S. K. Sahoo},
{\em Invariant Cassinian metrics}, Manuscript.


\bibitem{IMSZ}{Z. Ibragimov, M. R. Mohapatra, S. K. Sahoo, and X.-H. Zhang,}
{\em Geometry of the Cassinian metric and its inner metric},
Bull. Malays. Math. Sci. Soc., {\bf 40} (1) (2017), 361--372.


\bibitem{KMS} {R. Kl\'en, M. R. Mohapatra, and S. K. Sahoo},
\emph{Geometric properties of the Cassinian metric},
Math. Nachr. 1--13 (2016), {\tt DOI: 10.1002/mana.201600117}.

\bibitem{Sei99}  {P. Seittenranta,}
{\em  M\"obius-invariant metrics},
Math. Proc. Cambridge Philos. Soc., {\bf 125} (1999), 511--533.

\bibitem{Vai05}
{J. V\"ais\"al\"a},
{\em Gromov hyperbolic spaces}, 
Expo. Math., {\bf 23} (2005), 187--231.


\bibitem{Vuo85} {M. Vuorinen,}
{\em Conformal invariants and quasiregular mappings},
J. Anal. Math., {\bf 45} (1985), 69--115.

\bibitem{Vuo88} {M. Vuorinen,}
{\em Conformal geometry and quasiregular mappings},
Lecture note in mathematics, Springer-Verlag,
Berlin, 1988.

\end{thebibliography}
\end{document}